\documentclass[12pt]{amsart}

\usepackage{amssymb}
\usepackage{bm}
\usepackage{graphicx}
\usepackage{psfrag}
\usepackage[usenames,dvipsnames]{xcolor}
\usepackage{enumerate}
\usepackage{multirow}
\usepackage{url}

\usepackage{comment}

\usepackage{algpseudocode}

\usepackage{mathtools}

\usepackage{xy}
\input xy
\xyoption{all}

\newcommand{\excise}[1]{}

\newtheorem{thm}{Theorem}[section]
\newtheorem{lemma}[thm]{Lemma}

\newtheorem{cor}[thm]{Corollary}
\newtheorem{prop}[thm]{Proposition}

\theoremstyle{definition}

\newtheorem{example}[thm]{Example}
\newtheorem{remark}[thm]{Remark}
\newtheorem{defn}[thm]{Definition}

\numberwithin{equation}{section}

\renewcommand\>{\rangle}
\newcommand\<{\langle}

\newcommand\NN{\mathbb{N}}

\newcommand\RR{\mathbb{R}}

\newcommand\ZZ{\mathbb{Z}}

\newcommand\til{\mathord\sim}

\DeclareMathOperator\Betti{Betti} 

\DeclareMathOperator\LD{LD} 

\DeclareMathOperator\lcm{lcm} 

\begin{document}

\mbox{}
\title[Length density and numerical semigroups]{Length density and numerical semigroups}

\author[Brower ET AL.]{Cole Brower}
\address{Mathematics Department\\San Diego State University\\San Diego, CA 92182}
\email{cole.brower@gmail.com}

\author[]{Scott Chapman}
\address{Department of Mathematics and Statistics\\Sam Houston State University\\Huntsville, TX  77341}
\email{scott.chapman@shsu.edu}

\author[]{Travis Kulhanek}
\address{Mathematics Department\\University of California Los Angeles\\Los Angeles, CA 90095}
\email{volcanek47@gmail.com}

\author[]{Joseph McDonough}
\address{Mathematics Department\\San Diego State University\\San Diego, CA 92182}
\email{joe.mcdonough@live.com}

\author[]{Christopher O'Neill}
\address{Mathematics Department\\San Diego State University\\San Diego, CA 92182}
\email{cdoneill@sdsu.edu}

\author[]{Vody Pavlyuk}
\address{Mathematics Department\\San Diego State University\\San Diego, CA 92182}
\email{vody.pa@gmail.com}

\author[]{Vadim Ponomarenko}
\address{Mathematics Department\\San Diego State University\\San Diego, CA 92182}
\email{vponomarenko@sdsu.edu}

\date{\today}

\begin{abstract}
Length density is a recently introduced factorization invariant, assigned to each element $n$ of a cancellative commutative atomic semigroup $S$, that measures how far the set of factorization lengths of $n$ is from being a full interval.  We examine length density of elements of numerical semigroups (that is, additive subsemigroups of the non-negative integers).  
\end{abstract}

\maketitle


\section{Introduction}
\label{sec:intro}

A numerical semigroup $S$ is an additively closed subset of $\ZZ_{\ge 0}$ containing 0, usually specified using a generating set $n_1, \ldots, n_k$, i.e.,
$$S = \<n_1, \ldots, n_k\> = \{z_1n_1 + z_2n_2 + \cdots + z_kn_k \mid z_1, \ldots, z_k \in \ZZ_{\ge 0}\}.$$
A \emph{factorization} of an element of $S$ is an expression of the form $z_1n_1 + \cdots + z_kn_k$.  
Many classical problems surrounding factorizations in semigroup theory involve so-called factorization invariants, which are arithmetic quantities, often combinatorial in nature, that capture some precise aspect of non-uniqueness.  For instance, the elasticity invariant equals the quotient of the maximum and minimum factorization lengths of an element $n \in S$, and the delta set invariant $\Delta_S(n)$ contains the ``gaps'' in the set of possible factorization lengths of $n$.  

A recent paper \cite{lengthdensity} introduced a new invariant, known as \emph{length density}, that measures the sparse-ness of the length set, defined in such a way that $\LD_S(n) = 1$ if and only if the length set of $n$ is a full interval.  
The length density of any numerical semigroup $S$ has the following immedate upper and lower bounds \cite[Proposition 2.3]{lengthdensity}:
$$\frac{1}{\max\Delta(S)} \le \LD(S) \le \frac{1}{\min\Delta(S)}.$$
It is not hard to see that the second inequality above is strict precisely when $|\Delta(S)| > 1$.  The strictness of the first inequality, on the other hand, turns out to be more nuanced.  With this in mind, we introduce and study the following property:\ we say $S$ is 
\emph{tasty}\footnote{This term was chosen since the McNugget semigroup is tasty; see Example~\ref{e:mcnugdensity}.}
if $\LD(S) > 1/\max\Delta(S)$, and \emph{bland} otherwise.  In order for $S$ to be bland, some $n \in S$ must have $\Delta_S(n) = \{\max\Delta(S)\}$, a phenomenon that has been studied in the context of Krull monoids~\cite{subdeltas}.  As such, determining when $S$ is tasty necessitates sufficient control over the delta sets of elements of $S$.  

The aim of the present manuscript is to investigate how common the tasty and bland properties are among numerical semigroups. 
Our results are as follows.  After reviewing the necessary background in Section~\ref{sec:background}, we provide in Section~\ref{sec:asymptotic} an algorithm for computing $\LD(S)$ for any numerical semigroup $S$, and characterize the asymptotic behavior of length density for large elements of $S$.  In the remaining sections, we examine length density, and in particular the tastiness property, for several well-studied families of numerical semigroups.  Several interesting consequences are also obtained.  
\begin{itemize}
\item 
Any numerical semigroup with at most 3 generators has its length density achieved at a Betti element (Theorem~\ref{t:3genbettidensity}).  Note that a 4-generated numerical semigroup without this property is known (Example~\ref{e:notacceptedatbetti}).  

\item 
If $S$ has maximal embedding dimension and prime multiplicity, then the delta set of any Betti element containing $\max\Delta(S)$ must be a singleton.  This resembles a property of certain Krull monoids~\cite{subdeltas} that generally does not hold for numerical semigroups.  

\end{itemize}

\section{Background}
\label{sec:background}

We open with a series of definitions.

\begin{defn}\label{d:numerical}
A \emph{numerical semigroup} is a subset of $\ZZ_{\ge 0}$ of the form
$$S = \<n_1, \ldots, n_k\> = \{z_1n_1 + \cdots + z_kn_k : z_1, \ldots, z_k \in \ZZ_{\ge 0}\}$$
for the semigroup generated by $n_1, \ldots, n_k$.  
A \emph{factorization} of $n \in S$ is an expression 
$$n = z_1n_1 + \cdots + z_kn_k$$
of $n$ as a sum of generators of $S$, and the \emph{length} of a factorization is the sum $z_1 + \cdots + z_k$.  The \emph{set of factorizations} of $n$ is the set
$$\mathsf Z_S(n) = \{z \in \ZZ_{\ge 0}^k : n = z_1n_1 + \cdots + z_kn_k\}$$
viewed as a subset of $\ZZ_{\ge 0}^k$, and the \emph{length set} of $n$ is the set 
$$\mathsf L_S(n) = \{z_1 + \cdots + z_k : z \in \mathsf Z_S(n)\},$$
of all possible factorization lengths of $n$.  Writing $\mathsf L_S(n) = \{\ell_1 < \cdots < \ell_r\}$, define 
$$\Delta_S(n) = \{\ell_i - \ell_{i-1} : 2 \le i \le r\} \qquad \text{and} \qquad \Delta(S) = \bigcup_{n \in S} \Delta_S(n)$$
as the \emph{delta sets} of $n$ and $S$, respectively.  
\end{defn}

\begin{defn}\label{d:minimalpresentations}
Fix a numerical semigroup $S = \<n_1, \ldots, n_k\>$.  
The \emph{factorization homomorphism} of $S$ is the function $\varphi_S:\ZZ_{\ge 0} \to S$ given by
$$\varphi_S(z_1, \ldots, z_k) =  z_1n_1 + \cdots + z_kn_k$$
sending each $k$-tuple to the element of $S$ it is a factorization of.  The \emph{kernel} of $\varphi_S$ is the equivalence relation $\til = \ker \varphi_S$ that sets $z \sim z'$ whenever $\varphi_S(z) = \varphi_S(z')$.  The kernel $\til$ is in fact a \emph{congruence}, meaning that $z \sim z'$ implies $(z + z'') \sim (z' + z'')$ for all $z, z', z'' \in \NN^k$.  A subset $\rho \subset \ker\varphi_S$, viewed as a subset of $\NN^k \times \NN^k$, is a \emph{presentation} of $S$ if $\ker\varphi_S$ is the smallest congruence on $\NN^k$ containing $\rho$.  We say $\rho$ is a \emph{minimal presentation} if it is minimal with respect to containment among presentations for $S$.  
\end{defn}

\begin{defn}\label{d:factorizationgraph}
The \emph{factorization graph} of an element $n$ of a numerical semigroup $S$ is the graph $\nabla_n$ whose vertices are the elements of $\mathsf Z_S(n)$ in which two factorizations $z, z' \in \mathsf Z_S(n)$ are connected by an edge if $z_i > 0$ and $z_i' > 0$ for some $i$.  If $\nabla_n$ is disconnected, we say $n$ is a \emph{Betti element} of $S$, and we write $\Betti(S)$ for the set of Betti elements of $S$.  
\end{defn}

The following theorem is a summary of multiple results which can be found in \cite[Section 5.3]{AAG}.

\begin{thm}\label{t:allminimalpresentations}
A presentation $\rho$  of a finitely generated semigroup $S = \<n_0, \ldots, n_k\>$ is minimal if and only if for every $n \in S$, (i) the number of connected components in the graph $\nabla_n$ is one more than the number of relations in $\rho$ containing factorizations of $n$, and (ii) adding an edge to $\nabla_n$ corresponding to each such relation in $\rho$ yields a connected graph.  
\end{thm}

Our next definition yields our main property of interest.

\begin{defn}\label{d:lengthdensity}
Fix a numerical semgiroup $S = \<n_1, \ldots, n_k\>$.  The \emph{length density} of $n \in S$ with $|\mathsf L_S(n)| \ge 2$ is given by
$$\LD_S(n) = \frac{|\mathsf L(n)| - 1}{\max\mathsf L_S(n) - \min\mathsf L_S(n)}$$
and the \emph{length density} of $S$ is 
$$\LD(S) = \inf\{\LD_S(n) : n \in S, \, |\mathsf L_S(n)| \ge 2\}.$$
It is not hard to prove (see~\cite{lengthdensity}) that
$$\frac{1}{\max\Delta(S)} \le \LD(S) \le \frac{1}{\min\Delta(S)}.$$
We say $S$ is \emph{tasty} if 
$$\LD(S) > \frac{1}{\max\Delta(S)}$$
and \emph{bland} otherwise.  
\end{defn}

Before moving forward, we list a few basic results from \cite{lengthdensity} concerning the length density.

\begin{thm}[{\cite[Theorem 3.4 and Proposition 3.2]{lengthdensity}}]
For any numerical semigroup $S$, some $n \in S$ satisfies $\LD_S(n) = \LD(S)$.  Additionally, $S$ is bland if and only if 
$$\LD_S(b) = \frac{1}{\max\Delta(S)}$$
for some $b \in \Betti(S)$, which in particular occurs if $\Delta_S(b) = \{\max\Delta(S)\}$.  
\end{thm}


\section{Asymptotics and computation}
\label{sec:asymptotic}

In this section, we characterize asymptotic behavior of $\LD_S(n)$ for a given numerical semigroup $S$.  Like many other factorization invariants~\cite{factorhilbert,numericalsurvey}, for large $n$ the function $\LD_S(n)$ coincides with a quotient of quasilinear functions of $n$ (Theorem~\ref{t:ldquasi}).  The primary consequence is an explicit upper bound on the smallest element of $S$ with length density $\LD(S)$ (Corollary~\ref{c:ldquasi}), and thus an algorithm to compute $\LD(S)$ (Remark~\ref{r:computeld}).  

\begin{example}\label{e:mcnugdensity}
Figure~\ref{f:mcnugdensity} depicts the function $\LD_S(n)$ for $S = \<6, 9, 20\>$.  Note in particular that $S$ is tasty since $\LD(S) = \frac{4}{7} > 1/\max\Delta(S) = \tfrac{1}{4}$.  
\end{example}

\begin{figure}
\begin{center}
\includegraphics[height=2.5in]{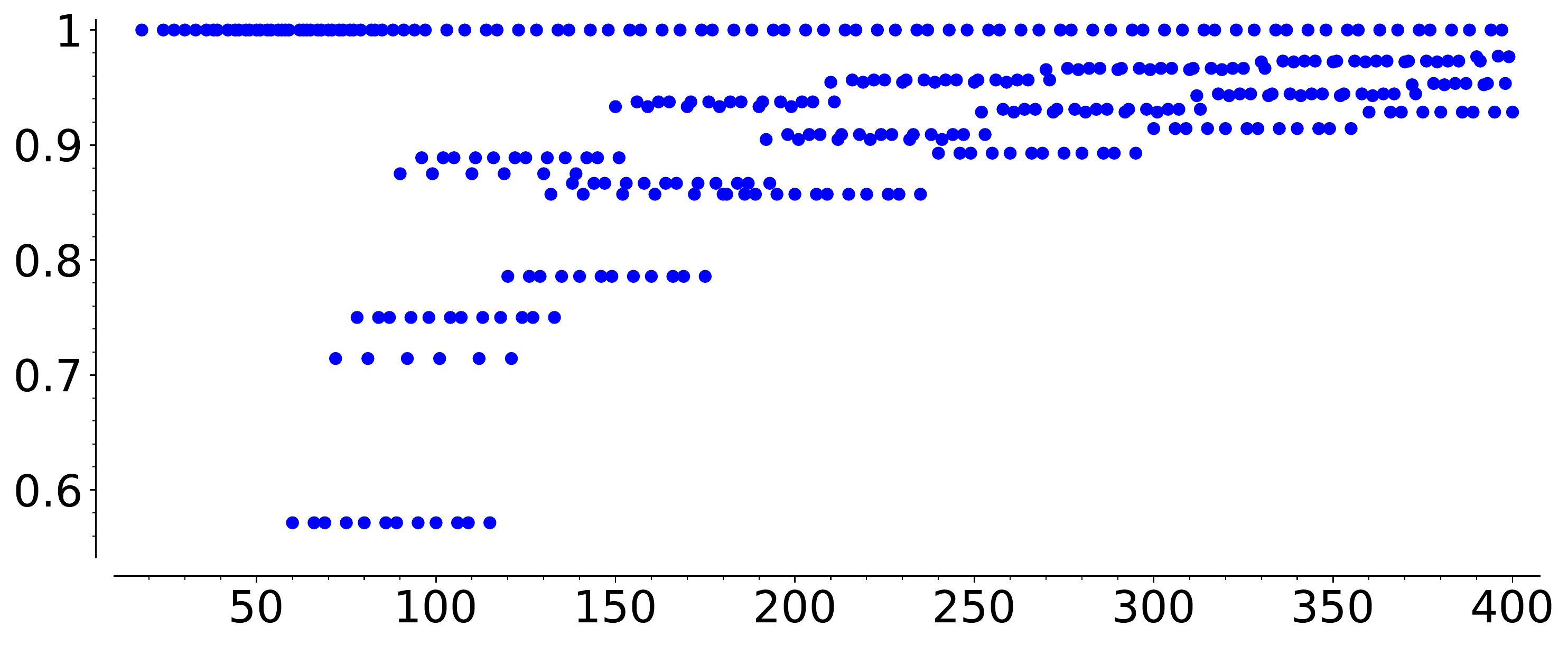}
\end{center}
\caption{A plot with a point at $(n, \LD_S(n))$ for each $n \in S = \<6, 9, 20\>$.}
\label{f:mcnugdensity}
\end{figure}

Throughout this section, suppose $S = \<n_1 < \cdots < n_k\>$ is a numerical semigroup with $\gcd(n_1, \ldots, n_k) = 1$, and let
$$
d = \min \Delta(S)
\qquad \text{and} \qquad
L = \frac{n_k - n_1}{d \gcd(n_1, n_k)}.
$$
Note $L \in \ZZ$ since $d = \gcd(n_2 - n_1, \ldots, n_k - n_{k-1})$ by \cite[Proposition~2.9]{delta}.  

We begin by recalling some pertinent results from~\cite{compasympdelta}, wherein the authors identify a constant~$N_S$, given as a (large) formula in terms of $n_1, \ldots, n_k$, such that 
$$\Delta_S(n + \lcm(n_1, n_k)) = \Delta_S(n).$$
for all $n \ge N_S$ \cite[Corollary~14]{compasympdelta}.  Letting
\begin{align*}
\mathsf L_1(n) &= \{\ell \in \mathsf L_S(n) \mid \tfrac{1}{n_1}n + N_S(\tfrac{1}{n_2} - \tfrac{1}{n_1}) < \ell \le \tfrac{1}{n_1}n\}, \\
\mathsf L_2(n) &= \{\ell \in \mathsf L_S(n) \mid \tfrac{1}{n_k}n + N_S(\tfrac{1}{n_{k-1}} - \tfrac{1}{n_k}) \le \ell \le \tfrac{1}{n_1}n + N_S(\tfrac{1}{n_2} - \tfrac{1}{n_1})\}, \\
\mathsf L_3(n) &= \{\ell \in \mathsf L_S(n) \mid \tfrac{1}{n_k}n \le \ell < \tfrac{1}{n_k}n + N_S(\tfrac{1}{n_{k-1}} - \tfrac{1}{n_k})\}
\end{align*}
so that 
$$\mathsf L_S(n) = \mathsf L_1(n) \cup \mathsf L_2(n) \cup \mathsf L_3(n)$$
is a disjoint union, the following also follows directly from the work in~\cite{compasympdelta}.  

\begin{lemma}\label{l:L1L3}
If $n \ge N_S$, then 
$$
\mathsf L_1(n + n_1) = \{\ell + 1 \mid \ell \in \mathsf L_1(n)\}
\qquad \text{and} \qquad
\mathsf L_3(n + n_k) = \{\ell + 1 \mid \ell \in \mathsf L_3(n)\},
$$
and in particular, 
$$
|\mathsf L_1(n + n_1)| = |\mathsf L_1(n)|, 
\qquad
|\mathsf L_3(n + n_k)| = |\mathsf L_3(n)|,
\qquad \text{and} \qquad
\Delta(\mathsf L_2(n)) = \{d\}.
$$
\end{lemma}

The following appeared as \cite[Theorem~4.5]{factorhilbert} but without an explicit lower bound.  We provide here a proof that $N_S$ is such a lower bound.  

\begin{thm}\label{t:lenquasi}
If $n \ge N_S$, then 
$$|\mathsf L(n + \lcm(n_1, n_k))| = |\mathsf L(n)| + L.$$
\end{thm}

\begin{proof}
First, we set $p = \lcm(n_1, n_r)$.  Using Lemma~\ref{l:L1L3}, we have
\begin{align*}
|\mathsf L(n + p)|
&= |\mathsf L_1(n + p)| + |\mathsf L_2(n + p)| + |\mathsf L_3(n + p)| \\
&= |\mathsf L_1(n)| + |\mathsf L_2(n + p)| + |\mathsf L_3(n)| \\
&= |\mathsf L(n)| + |\mathsf L_2(n + p)| - |\mathsf L_2(n)| \\
&= |\mathsf L(n)| + \tfrac{1}{d}(\max\mathsf L_2(n + p) - \min\mathsf L_2(n + p)) - \tfrac{1}{d}(\max\mathsf L_2(n) - \min\mathsf L_2(n)) \\
&= |\mathsf L(n)| + \tfrac{1}{d}(\min\mathsf L_1(n + p) - \max\mathsf L_3(n + p)) - \tfrac{1}{d}(\min\mathsf L_1(n) - \max\mathsf L_3(n)) \\
&= |\mathsf L(n)| + \tfrac{1}{d}(\min\mathsf L_1(n) - \max\mathsf L_3(n) + dL) - \tfrac{1}{d}(\min\mathsf L_1(n) - \max\mathsf L_3(n)) \\
&= |\mathsf L(n)| + L
\end{align*}
as desired.
\end{proof}

\begin{thm}\label{t:ldquasi}
If $n \ge N_S$, then 
$$
\LD(n + \lcm(n_1, n_k)) = \frac{|\mathsf L(n)| - 1 + L}{\max\mathsf L(n) - \min\mathsf L(n) + dL}.
$$
In particular, $\LD(n)$ is a quotient of eventually quasilinear functions of $n$, each with period dividing $\lcm(n_1, n_k)$.
\end{thm}

\begin{proof}
First, let $p = \lcm(n_1, n_k)$.  It follows from \cite[Theorems~4.2 and~4.3]{elastsets} that
$$
\max\mathsf L(n + n_1) = \max\mathsf L(n) + 1
\qquad \text{and} \qquad
\min\mathsf L(n + n_k) = \min\mathsf L(n) + 1
$$
hold for all $n \ge n_k^2$, and since $N_S > n_k^2$, one readily checks that
\begin{align*}
\LD(n + p)
&= \frac{|\mathsf L(s + p)| - 1}{\max(\mathsf L(s + p)) - \min(\mathsf L(s + p))} \\
&= \frac{|\mathsf L(s)| - 1 + L}{(\max(\mathsf L(s)) + \tfrac{1}{n_1}p) - (\min(\mathsf L(s)) + \tfrac{1}{n_k}p)} \\
&= \frac{|\mathsf L(s)| - 1 + L}{\max(\mathsf L(s)) - \min(\mathsf L(s)) + dL}
\end{align*}
holds for all $n \ge N_S$.  
\end{proof}

Our final result of the section is a corollary of Theorem~\ref{t:ldquasi} that is topological in nature, reminiscent of \cite[Theorems~2.1-2.2 and Corollary~2.3]{fullelastic} and \cite[Corollary~4.5]{elastsets} concerning the set of elasticities of a numerical semigroup.  

\begin{cor}\label{c:ldquasi}
We have
$$\LD(S) = \min\{\LD(n) \mid n \in S \text{ and } n < N_S + \lcm(n_1, n_k)\}.$$
Additionally, letting $R(S) = \{\LD(n) : n \in S\}$, the set $R(S) \cap [0,\alpha)$ is finite for each $\alpha \in [0, \tfrac{1}{d}]$, and the only possible accumulation point of $R(S)$ is 
$$\sup R(S) = \lim_{n \to \infty} \LD(n) = \tfrac{1}{d}.$$
\end{cor}

\begin{proof}
All 3 claims follow from Theorem~\ref{t:ldquasi} and the observation that
$$\LD(n + \lcm(n_1, n_k)) \ge \LD(n)$$
for all $n \ge N_S$.  
\end{proof}

\begin{remark}\label{r:computeld}
Corollary~\ref{c:ldquasi} yield a method of computing $\LD(S)$ from the generators of $S$.  Indeed, $N_S$ can be immeidately computed from the formula in~\cite[Section~3]{compasympdelta}, and the length sets of all $n \le N_S + \lcm(n_1, n_k)$ can be computed relatively quickly using the methods in \cite[Section~3]{dynamicalg}.  Additionally, Theorem~\ref{t:ldquasi} yields an algorithm to compute $\LD_S(n)$ for $n \ge N_S$ whose runtime does not depend on $n$, since one only needs to compute the length set of an appropriate element between $N_S$ and $N_S + \lcm(n_1, n_k)$.  
\end{remark}

\section{Families of numerical semigroups}
\label{sec:families}

In this section, we classify the length density and tastiness of numerical semigroups residing in one of several well-studied families.

\subsection{Supersymmetric numerical semigroups}

A numerical semigroup is called \emph{supersymmetric} if it has exactly one Betti element.  
By \cite[Theorem~12]{uniquebetti}, this occurs if and only if $S$ can be written in the form
$$S = \< \tfrac{s}{t_1}, \tfrac{s}{t_2}, \ldots, \tfrac{s}{t_k}\>$$
for some pairwise coprime $t_1 > \cdots > t_k$ with $s = t_1t_2 \cdots t_k$.  Moreover, $s$ is the unique Betti element in this case, and has length set $\mathsf L(s) = \{t_k, \ldots, t_1\}$.  

\begin{lemma} \label{l:supersymmetric}
For any $m \ge 1$, $|\mathsf L(ms)| \ge mk - m + 1$.
\end{lemma}

\begin{proof}
We proceed by induction on $m$.  For the base case, notice that 
$$|\mathsf L(s)| = |\{t_k, \ldots, t_1\}| \ge (1)k - (1) + 1.$$
Next, assuming $|\mathsf L(ms)| \ge mk - m + 1$, we want to show 
$$|\mathsf L((m+1)s)| \ge (m+1)n - (m+1) + 1.$$
Since $(m+1)s = ms + s$, it follows that 
$$T := \mathsf L(ms) + t_k \subseteq \mathsf L((m+1)s)$$
has at least $mn - m + 1$ elements by our inductive hypothesis.  Let $\ell$ denote the greatest element of $\mathsf L(ms)$, and fix a factorization $(a_1, \ldots, a_k)$ of $ms$ of length $\ell$. Then for each $i$, $(a_1, \ldots, a_i + t_i, \ldots, a_k)$ is a length $\ell + t_i$ factorization of $(m+1)s$.  Since $\ell + t_k = \max T$, 
$$\ell + t_1, \ldots, \ell + t_{k-1} \in \mathsf L((m+1)s) \setminus T,$$
from which we conclude 
$$\left| \mathsf L((m+1)s) \right| \geq (mk - m + 1) + k - 1 = (m+1)k - (m+1) + 1,$$
as desired. 
\end{proof}

\begin{thm}\label{t:supersymmetric}
If $n \in S$ has at least 2 factorizations, then $\LD(n) \geq \LD(s)$.  Moreover, 
$$\LD(S) = \LD(s) = \frac{k-1}{t_1-t_k},$$
and $S$ is tasty if and only if $ t_1, t_2, \ldots, t_k$ does not form an arithmetic progression.  
\end{thm}

\begin{proof}
By~\cite[Theorem~12]{uniquebetti}, we can write $n = ms + r$ with $m \ge 1$ such that $r \in S$ has unique factorization, say with length $\ell$.  In this case, $\mathsf L(n) = \mathsf L(ms) + \ell$.  Since $\max\mathsf L(ms) = mt_1$ and $\min \mathsf L(ms) = mt_n$, Lemma~\ref{l:supersymmetric} yields
$$\LD(n) = \LD(ms)
= \frac{\left| \mathsf L(ms) \right| - 1}{mt_1 - mt_k}
\ge \frac{(mk - m +1) - 1}{mt_1 - mt_k} = \frac{k-1}{t_1-t_k} = \LD(s).$$
The final claim now follows from the fact that $\Delta(s)$ is a singleton if and only if 
$$\mathsf L(s) = \{t_k, \ldots, t_1\}$$
forms an arithmetic progression.  
\end{proof}

\subsection{Embedding dimension 3 numerical semigroups}

Numerical semigroups with 3 generators have been well-studied, and it is known that their factorization structure is largely determined by whether there are 1, 2, or 3 Betti elements.  Throughout this subsection, we use the notation from~\cite[Chapter~9]{numerical}, and refer the reader there for a thorough overview of this dichotomy. 

\begin{prop}\label{p:3gen3betti}
If $S = \<n_1,n_2,n_3\>$ has $3$ Betti elements, then $S$ is bland.
\end{prop}

\begin{proof}
Each $b \in \Betti(S)$ has exactly 2 factorizations, which, after relabeling the generators accordingly, have the form
$$b = c_1n_1 = r_2n_2 + r_3n_3.$$
As such, either $\Delta(b) = \emptyset$ or $\Delta(b) = \{\delta\}$, where $\delta = |r_2 + r_3 - c_1|$.  Thus, $S$ is bland.  
\end{proof}

\begin{prop}\label{p:3gen2betti}
If $S = \<n_1,n_2,n_3\>$ with $\Betti(S) = \{b_1, b_2\}$ and $b_1 < b_2$, then $S$ is tasty if and only if $\max\Delta(b_2) > \max\Delta(b_1)$ and $b_2 - b_1 \in S$.
\end{prop}

\begin{proof}
For the backward direction, suppose $\max\Delta(b_2) > \max\Delta(b_1)$ and $b_2 - b_1 \in S$.  We must have $\max\Delta(S) = \max\Delta(b_2)$ by \cite[Theorem~2.5]{deltabetti}, but since $b_2 - b_1 \in S$, the trade defined at $b_1$ can be used at $b_2$ to obtain an element of $\Delta(b_2)$ at most $\max\Delta(b_1)$.  As such, $|\Delta(b_2)| \ge 2$, so $S$ is tasty.

Conversely, if $\max\Delta(b_2) \le \max\Delta(b_1)$, then $S$ is bland since $|\Delta(b_1)| = 1$, and if $b_2 - b_1 \notin S$, then $b_1$ and $b_2$ both have singleton delta sets, so again $S$ is bland.  This completes the proof.  
\end{proof}

\begin{thm}\label{t:3genbettidensity}
If $S = \<n_1, n_2, n_3\>$, then $\LD(S)$ occurs at a Betti element.  
\end{thm}

\begin{proof}
If $S$ has 1 Betti element, then $S$ is supersymmetric, so apply Theorem~\ref{t:supersymmetric}.  The~claim clearly holds if $S$ is bland, so by Propositions~\ref{p:3gen3betti} and~\ref{p:3gen2betti} it suffices to assume $\Betti(S) = \{b_1, b_2\}$, $b_2 - b_1 \in S$, and $\max\Delta(b_2) > \max\Delta(b_1)$.  Write $S = \<n_1, n_2, n_3\>$,
$$b_1 = c_1n_1 = c_2n_2
\qquad \text{and} \qquad
b_2 = c_3n_3 = r_1n_1 + r_2n_2,$$
and let
$$\delta_1 = c_1 - c_2
\qquad \text{and} \qquad 
\delta_2 = r_1 + r_2 - c_3.$$
Notice $t_1 = (c_1, -c_2, 0)$ and $t_2 = (r_1, r_2, -c_3)$ form a minimal presentation for $S$.  Assume $c_1 > c_2$, and the two given factorizations of $b_2$ have extremal lengths in $\mathsf L(b_2)$ (or, equivalently, that $|\delta_2|$ is maximal among choices for the trade $t_2$).  

Fix $n \in S$ not uniquely factorable.  If $n - b_2 \notin S$, then only the trade $t_1$ is available, so $\Delta(n) = \{\delta_1\}$ and thus $\LD(n) = \LD(b_1)$.  As such, suppose $n - b_2 \in S$.  Fix a factorization $z = (z_1, z_2, z_3) \in \mathsf Z(n)$ with $z_3$ maximal, and write $z_3 = qc_3 + r$ with $q, r \in \ZZ_{\ge 0}$ and $r < c_3$.  Since $t_2$ is the only trade involving $n_3$, we must have
$$\{y_3 : (y_1, y_2, y_3) \in \mathsf Z(n)\} = \{r, r + c_3, \ldots, r + qc_3\}.$$
Moreover, let $\ell = z_1 + z_2 + z_3 - c_3$, and let
$$L = \mathsf L(b_2) + \{\ell, \ell + \delta_2, \ldots, \ell + (q-1)\delta_2\}.$$
By repeatedly performing the trade $t_2$ to $z$, we see $L \subseteq \mathsf L(n)$, and by the maximality of $|\delta_2|$, $L$ is a union of translations of $\mathsf L(b_2)$ with matching endpoints.  In particular, 
$$\LD(L) = \LD(b_2) > \LD(b_1) = \tfrac{1}{\delta_1}.$$
The key obsevation is now that by the maximality of $|\delta_2|$, every length in $\mathsf L(n) \setminus L$ is attained by a factorization obtained from performing the trade $t$, perhaps multiple times, to some factorization whose length lies in $L$.  As such, letting $m = |\mathsf L(n) \setminus L|$, we obtain
$$\LD(n) \ge \frac{|L| + m}{\max L - \min L + m\delta_1} \ge \frac{|L|}{\max L - \min L} = \LD(b_2),$$
as desired.  
\end{proof}

\begin{example}\label{e:notacceptedatbetti}
The conclusion of Theorem~\ref{t:3genbettidensity} can fail if $S$ has 4 or more generators.  Indeed, the semigroup $S = \<20, 28, 42, 73\>$, which appeared as \cite[Example~3.3]{lengthdensity}, has Betti element length sets $\mathsf L(84) = \{2, 3\}$, $\mathsf L(140) = \{4, 5, 7\}$, and $\mathsf L(146) = \{2, 4, 5\}$, but a strictly smaller length density results from $\mathsf L(202) = \{4, 6, 7, 9\}$.  
\end{example}


\subsection{Maximal embedding dimension numerical semigroups}

We say $S$ is \emph{maximal embedding dimension} (or \emph{MED}) if $e(S) = m(S)$.  In this subsection, we write 
$$S = \<m, n_1, \ldots, n_{m-1}\>$$
and assume $n_i \equiv i \bmod m$ for each $i$.  It is known in this case that 
$$\Betti(S) = \{n_i + n_j : 1 \le i \le j \le m-1\};$$
we refer the reader to \cite[Section~7.4]{numerical}.  

We begin by identifying a class of Betti elements whose delta sets are singletons.  

\begin{prop}\label{p:medsmallbetti}
Suppose $b \in \Betti(S)$ is the smallest Betti element in the equivalence class of $k \in [0, m-1]$ modulo $m$.  If $k = 0$ or $k$ is a unit in $\ZZ_m$, then $|\mathsf L(b)| \le 2$.  
\end{prop}

\begin{proof}
First, suppose $k = 0$, so that $b$ is the smallest Betti element of $S$ divisible by $m$.  
We claim any factorization of $b$ involving $n_i$ with $1 \le i \le m-1$ must have length~2 (and thus must be the factorization $n_i + n_{m-i}$).  Indeed, by the minimality of $b$, 
$$n_i + n_{m-i} - b = cm$$
for some $c \ge 0$, so any factorization of $b$ of length 3 or more involving~$n_i$ would yield a factorization of $n_{m-i}$ of length at least 2, which is impossible since $n_{m-i}$ is a minimal generator of $S$.  This proves $\mathsf L(b) = \{2, a\}$, where $b = am$.  

In all remaining cases, $b \equiv k \bmod m$ with $1 \le k \le m - 1$.  
Write
$$b = am + n_k = n_i + n_{k-i},$$
where $a \ge 2$ and $1 \le i \le m-2$.  By similar reasoning as above, any factorization involving $n_i$ with $i \ne k$ must be the factorization $n_i + n_{k-i}$ by the minimality of~$b$.   As~such, since $k$ is a unit in $\ZZ_m$, any factorization involving $n_k$ other than $am + n_k$ must contain $(m+1)n_k$, which is impossible since
$$(m + 1)n_k = cn_k + (m+1-c)n_k > n_i + n_{k-i} \ge b$$
for some $c$.  This again yields the desired claim.  
\end{proof}

\begin{thm}\label{t:medprime}
If $m \ge 3$ is prime, then $S$ is bland.  
\end{thm}

\begin{proof}
Fix $a \ge 1$ so that $am$ is the smallest Betti element that is a multiple of $m$.  
By~relabeling the generators of $S$ using an appropriately chosen automorphism of $\ZZ_m$, it suffices to assume $am = n_1 + n_{m-1}$.  
In what follows, we show
$$\max\Delta(S) = a - 2,$$
which completes the proof since $\Delta(am) = \{a - 2\}$ by Proposition~\ref{p:medsmallbetti}.  

First, fix $b' = n_j + n_{m-j} \in \Betti(S)$, and suppose $b' > am$.  
Let 
$$b = n_i + n_{m-i} \in \Betti(S)$$
in such a way that $i$ is maximal such that $i < j$ and $b < b'$ (note $i=1$ satisfies both constraints, so the set of eligible $i$ is nonempty).  By induction on $b'$, we can assume $\max\Delta(b) \le a - 2$.  Fix $c, c' \ge 1$ such that 
$$
b' = cm + n_i + n_{m-i}
\qquad \text{and} \qquad
n_{i+1} + n_{m-i-1} = c'm + n_i + n_{m-i}.
$$
Since $n_{i+1}$ is a minimal generator of $S$, $n_i + n_1 > n_{i+1}$ and thus $n_i + n_1 \ge n_{i+1} + m$.  Similarly, $n_{m-i} + n_{m-1} \ge n_{m-i-1} + m$, and we obtain
$$c'm = n_{i+1} + n_{m-i-1} - (n_i + n_{m-i}) \le n_1 + n_{m-1} - 2m = m(a - 2),$$
so by maximality of $i$, we have $c \le c' \le a - 2$.  As such, from
$$\{2\} \cup (c + 2 + \mathsf L(b)) \subseteq \mathsf L(b'),$$
we conclude $\max\Delta(b') \le a - 2$.  

Next, fix $k \in [1, m-1]$.   Letting 
$b = n_1 + n_{k-1} = cm + n_k,$
for some $c$, we see 
$$cm = n_1 + n_{k-1} - n_k \le n_1 + n_{m-1} - m = m(a - 1).$$
This means $\mathsf L(b) \subset [2, c+1] \subset [2, a]$, so $\max\Delta(b) \le a - 2$.  More generally, let $b' = n_j + n_{k-j}$ with $1 \le j \le k-j \le k$, and let
$b = n_i + n_{j-i}$ with $i$ maximal subject to $i < j$ and $b < b'$.  By inductionon $b'$, we can assume $\max\Delta(b) \le a - 2$.  Fix $c, c'$ so that
$$
b' = cm + n_i + n_{k-i}
\qquad \text{and} \qquad
n_{i+1} + n_{k-i-1} = c'm + n_i + n_{k-i}.  
$$
We obtain
$$c'm = n_{i+1} + n_{k-i-1} - (n_i + n_{k-i}) \le n_1 + n_{m-1} - 2m = m(a - 2),$$
so by the maximality of $i$, we have $c \le c' \le a - 2$.  As before, we have
$$\{2\} \cup (c + 2 + \mathsf L(b)) \subseteq \mathsf L(b'),$$
and can conclude $\max\Delta(b') \le a - 2$.  A similar argument for $b' = n_{k+j} + n_{m-j}$ with $k + 1 \le k + j \le m - j \le m - 1$ completes the proof.  
\end{proof}

\begin{prop}\label{p:medcomposite}
There exist both bland and tasty MED numerical semigroups of each composite multiplicity.  
\end{prop}

\begin{proof}
The numerical semigroup $S = \<m, m + 1, \ldots, 2m - 1\>$ is bland since it is arithmetical and thus has singleton delta set by \cite[Theorem~2.2]{setoflengthsets}.  

Now, if $m = pq$ with $p$~prime and $q \ge 2$, then we claim $S = \<m, n_1, \ldots, n_{m-1}\>$ with
$$n_i = \begin{cases}
m + i & \text{if } p \mid i; \\
2qm + i & \text{if } p \nmid i,
\end{cases}$$
is tasty.  Indeed, consider a Betti element $b = n_i + n_j$.  For convenience in what follows, we write $n_{i+j} = n_{i+j-m}$ if $i + j > m$.  If $p \mid i$, then $p \mid j$ if and only if $p \mid (i + j)$, so 
$$n_i + n_j = \begin{cases}
m + n_{i+j} & \text{if } i + j < m; \\
2m + n_{i+j-m} & \text{if } i + j > m,
\end{cases}$$
and thus $\mathsf L(b) \subseteq \{2,3\}$.  Alternatively, if $p \nmid i$ and $p \nmid j$, then $b = n_{i+j} + cm$ with $c \ge 2q$, so the trade $qn_p = (q+1)m$ can be performed at least once, meaning $\min\Delta(b) = 1$.  

The above argument implies each $b \in \Betti(S)$ has either $\Delta(b) = \emptyset$ or $1 \in \Delta(b)$.  Lastly, the Betti element 
$n_1 + n_{m-1} = (2B+1)m$
has no factorizations of length 3, since any such factorization can use only $m, n_p, \ldots, n_{m-p}$ but $3n_{m-p} < n_1 + n_{m-1}$.  Thus, $\max\Delta(S) > 1$ and occurs at a Betti element with non-singleton delta set.  
\end{proof}

We close this section by examining multiplicity 4 MED numerical semigroups, where geometry plays a role in determining whether each semigroup is bland or tasty.  

Given $n_1, n_2, n_3 > 4$ with $n_i \equiv i \bmod 4$ for each $i$, the semigroup $S = \<4, n_1, n_2, n_3\>$ is MED if and only if 
$$
2n_1 > n_2
\qquad
n_1 + n_2 > n_3,
\qquad
n_2 + n_3 > n_1, 
\qquad \text{and} \qquad
2n_3 > n_2
$$
each hold~\cite{kunz}.  As such, it is natural to representing each MED numerical semigroup $S = \<4, n_1, n_2, n_3\>$ as a point $(n_1,n_2,n_3) \in \RR^3$.  Examining semigroups with fixed coordinate sum $n_1 + n_2 + n_3$ yields a cross section as depicted on the left in Figure~\ref{f:m4plots}.  The two regions labeled ``bland'' coincide with the semigroups where $\min(n_1,n_2,n_3)$ equals $n_1$ and $n_3$, respectively (Proposition~\ref{p:med4firstgen}).  For the remaining semigroups, if $n_1 + n_3$ is sufficiently larger than $n_2$, then $S$ is guaranteed tasty by (Theorem~\ref{t:med4tasty}) and lies in the region labeled ``tasty'' on the left.  This phenomenon is also visible in the plot on the right in Figure~\ref{f:m4plots}, which depicts the cross section $n_2 = 18$, placing a large (red) point at $(n_1, n_3)$ if $S$ is bland and a smaller (black) point if $S$ is tasty.

\begin{figure}
\begin{center}
\includegraphics[height=2.1in]{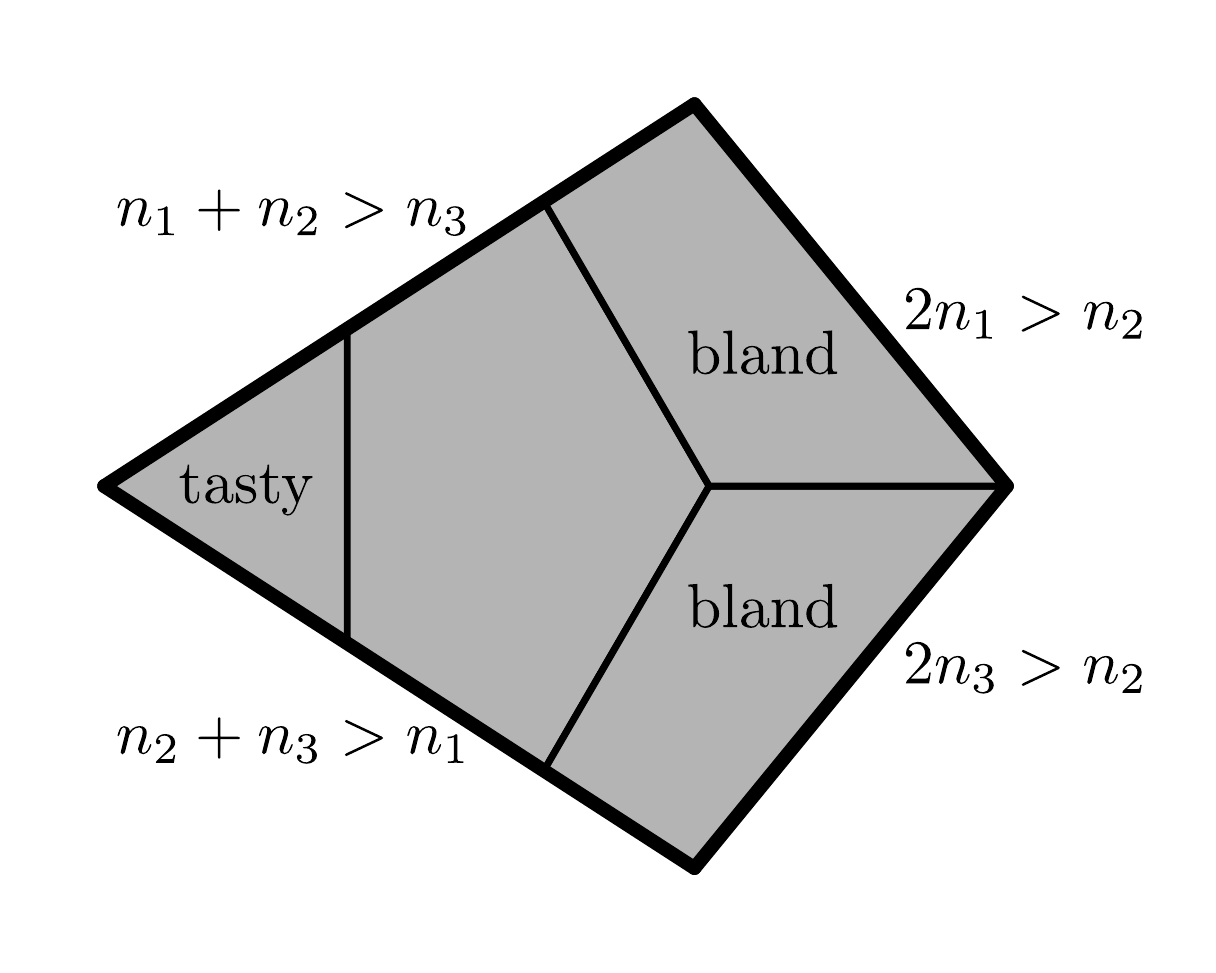}
\hspace{0.1in}
\includegraphics[height=2.1in]{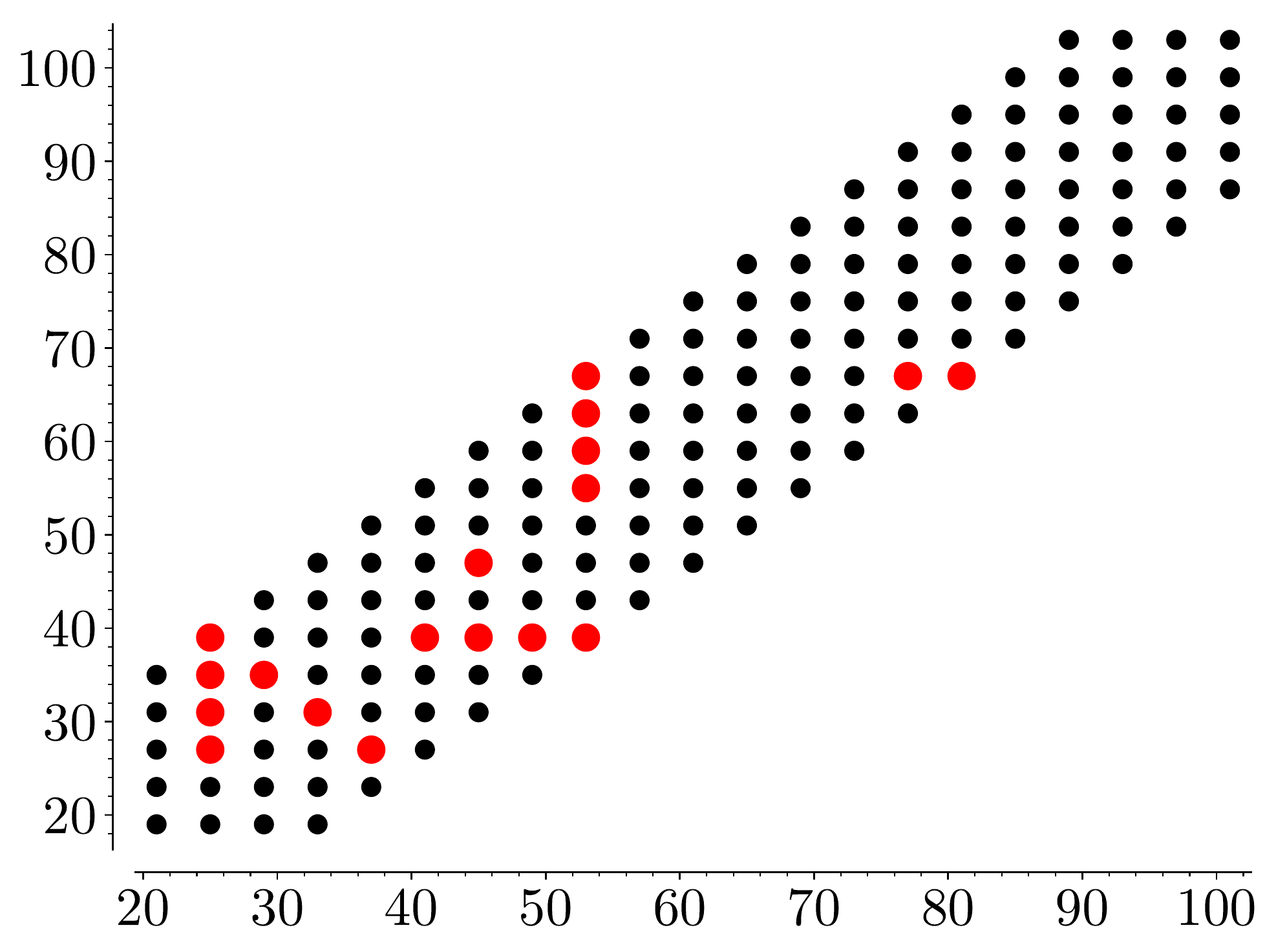}
\end{center}
\caption{A diagram (left) of a cross section of MED numerical semigroups $S = \<4, n_1, n_2, n_3\>$ with $n_1 + n_2 + n_3$ fixed, and a plot (right) obtained by setting $n_2 = 18$ and placing a large (red) point at $(n_1, n_3)$ if~$S$ is bland and a smaller (black) point if $S$ is tasty.}
\label{f:m4plots}
\end{figure}

We note that the geometric viewpoint discussed above, first outlined in \cite{kunz}, has proven fruitful in recent years for studying enumerative questions surrouding numerical semigroups; we direct the interested reader to \cite{wilfmultiplicity,kunzfaces1}.  

\begin{prop}\label{p:med4firstgen}
If $m = 4$ and $\min(n_1, n_2, n_3) \ne n_2$, then $S$ is bland. 
\end{prop}

\begin{proof}
Throughout this proof, we assume $n_1 = \min(n_1, n_2, n_3)$, as an analogous argument follows upon reversing the roles of $n_1$ and $n_3$ throughout.  

The Betti elements $n_1 + n_2$ and $n_2 + n_3$ are unique in their respective equivalence classes, and have singleton delta set by Proposition~\ref{p:medsmallbetti}.  Consider the Betti elements $b = 4a$ and $b' = 4(a + c)$ with $a, c \ge 0$.  Again by Proposition~\ref{p:medsmallbetti}, $b$ has singleton delta set, and for $b'$, there are two cases.  First, suppose
$$
b = 4a = n_1 + n_3
\qquad \text{and} \qquad
b' = 4(a + c) = 4c + n_1 + n_3 = 2n_2.
$$
Clearly $\mathsf L(b') \subseteq [2, a + c]$, so $\max\Delta(b') \le \max(c, a - 2)$.  We know $\Delta(b) = \{a-2\}$, and since $n_2$ is a minimal generator of $S$,
$$2n_1 \ge n_2 + 4
\qquad \text{and} \qquad
2n_3 \ge n_2 + 4,$$
so we obtain
$$4c = 2n_2 - (n_1 + n_3) \le n_1 + n_3 - 8 = 4(a - 2)$$
and thus $c \le a - 2$.  As such, $\max\Delta(b') \le \max\Delta(b)$.  
Second, suppose
$$
b = 4a = 2n_2
\qquad \text{and} \qquad
b' = 4(a + c) = 4c + 2n_2 = n_1 + n_3.
$$
Again, clearly $\mathsf L(b') \subseteq [2, a + c]$.  Since $n_2$ and $n_3$ are minimal generators of $S$, we have
$$2n_1 \ge n_2 + 4
\qquad \text{and} \qquad
n_1 + n_2 \ge n_3 + 4,$$
from which we obtain $3n_1 \ge n_3 + 8$.  From there, $n_1 < n_2$ implies
$$4c = n_1 + n_3 - 2n_2 \le 4n_1 - 2n_2 - 8 \le 2n_2 - 8 = 4(a - 2)$$
and we again conclude $c \le a - 2$ and $\max\Delta(b') \le \max\Delta(b)$.  

This leaves the Betti elements $2n_1$ and $2n_3$.  Write 
$$
2n_1 = n_2 + 4d
\qquad \text{and} \qquad
2n_3 = 2n_1 + 4e
$$
for $d, e \ge 1$.  
Now, since $n_1 < n_2$, we have
$4d = 2n_1 - n_2 < n_1 < 4a$,
so $d \le a - 1$.  As~such, since $\mathsf L(2n_1) \subset [2, d+1]$, we have 
$$\max\Delta(2n_1) \le a - 2 = \max\Delta(b).$$
Lastly, from $n_1 + n_2 \ge n_3 + 4$ and $3n_1 \ge n_3 + 8$, we obtain
$$
4e = 2n_3 - 2n_1 \le 2n_2 - 8
\qquad \text{and} \qquad
4e = 2n_3 - 2n_1 \le n_1 + n_3 - 8,
$$
so $e \le a - 2$ and thus $\max\Delta(2n_3) \le a - 2$ as well.  
\end{proof}

By examining where in the proof of Proposition~\ref{p:med4firstgen} the hypothesis $n_2 < n_1$ is used, we obtain the following.  

\begin{cor}\label{c:med4tastybettis}
If $m = 4$, then $\max\Delta(S)$ occurs at $\max(2n_2, n_1 + n_3)$ or $\min(2n_1, 2n_3)$.  
\end{cor}

\begin{proof}
By the proof of Proposition~\ref{p:med4firstgen}, it suffices to ensure $\max\Delta(S)$ does not occur~at
$$
b = n_1 + n_2 = 4d + n_3
\qquad \text{or} \qquad
b' = n_2 + n_3 = 4e + n_1.
$$
Writing $4a = \min(2n_2, n_1 + n_3)$, summing these equalities above yields
$$4d + 4e = 2n_2 \ge 4a,$$
so $\max\Delta(b) = d - 1 \le a - 2$ and $\max\Delta(b') = e - 1 \le a - 2$.  
\end{proof}

\begin{thm}\label{t:med4tasty}
Suppose $m = 4$ and $n_2 < n_1, n_3$.  If $2n_1 + 2n_3 > n_2^2$, then $S$ is tasty.  In particular, for fixed $n_2$, there are only finitely many bland numerical semigroups.  
\end{thm}

\begin{proof}
As in the proof of Proposition~\ref{p:med4firstgen}, let $2n_2 = 4a$ so that $\mathsf L(2n_2) = \{2, a\}$.  Let 
$$
b = n_1 + n_3 = 4c
$$
for $c \ge 1$, and write $c = qa + r$ with $0 \le r < a$.  Any~factorization of $b$ aside from $n_1 + n_3$ uses only $n_2$ and $4$, the shortest of which is 
$$4r + 2qn_2 = 4r + 4aq = 4c = b.$$
This means $\Delta(b) = \{a - 2, 2q + r - 2\}$, so $S$ is tasty if $2q + r \ne a$.  This is achieved by
$$4a(2q + r) = 8qa + 4ar > 8qa + 8r = 8c = 2n_1 + 2n_3 > n_2^2 = 4a^2$$
whenever $2n_1 + 2n_3 > n_2^2$.  

Supposing $n_1 < n_3$, by Corollary~\ref{c:med4tastybettis} it remains to~show
$$b' = 2n_1 = 4d + n_2$$
is either non-singleton or does not contain $\max\Delta(S)$.  Since $n_1 + n_2 \ge n_3 + 4$, we have 
$$8d = 4n_1 - 2n_2 \ge 2n_1 + 2n_3 > n_2^2 > 4n_2 = 8a,$$
meaning some element of $\Delta(b')$ is at most $a - 2$, as desired.  An analogous argument when $n_3 < n_1$ completes the proof.  
\end{proof}

\section{Tasty and bland gluings of numerical semigroups}
\label{sec:gluings}

Given two numerical semigroups $S_1 = \<n_1, \ldots, n_r\>$ and $S_2 = \<n_{r+1}, \ldots, n_k\>$ and non-atoms $\lambda \in S_1$ and $\mu \in S_2$ such that $\gcd(\lambda, \mu) = 1$, we say the numerical semigroup 
$$\mu S_1 + \lambda S_2 = \<\mu n_1, \ldots, \mu n_r, \lambda n_{r+1}, \ldots, \lambda n_k\>$$
is a \emph{gluing} of $S_1$ and $S_2$ by $\lambda$ and $\mu$.  It is known that the above generating set for $S$ is minimal, and that 
$$\Betti(S) = \mu\Betti(S_1) \cup \lambda\Betti(S_2) \cup \{\lambda \mu\}.$$
For more background on gluings, see~\cite[Chapter~8]{numerical}.  

In this section, we investigate the following question.  
\begin{quote}
If $S_1$ and $S_2$ are fixed and $\mu$ and $\lambda$ are allowed to vary, what can be said about how often the gluing $\mu S_1 + \lambda S_2$ is tasty, and how often it is bland?
\end{quote}

Our results are extremal in nature.  Theorem~\ref{t:infinitetastygluings} implies there will always be infinitely many tasty gluings $\mu S_1 + \lambda S_2$ for fixed $S_1$ and $S_2$, but the same does not hold for bland gluings.  In~particular, it is possible for two given numerical semigroups $S_1$ and $S_2$ to have infinitely many bland gluings (Theorem~\ref{t:infiniteblandgluings}), finitely many bland gluings (Theorem~\ref{t:singleblandgluing}), or no bland gluings (Theorem~\ref{t:noblandgluings}).  

\begin{thm}\label{t:infinitetastygluings}
Given numerical semigroups $S_1$ and $S_2$, there are infinitely many gluings $S = \mu S_1 + \lambda S_2$ that are tasty.
\end{thm}

\begin{proof}
Let $n_k$ denote the largest atom of $S_2$, and let $d = \max(\Delta(S_1) \cup \Delta(S_2))$.  First, choose any prime $\lambda \in S_1$ such that $|\mathsf L_{S_1}(\lambda)| \ge 2$ and $\lambda > n_2$.  Next, choose any prime $p > \lambda$ such that $$p \ge \max \mathsf L_{S_1}(\lambda) + d + 1$$
and let $\mu = pn_k$ so that $\min \mathsf L_{S_2}(\mu) = p$.  Clearly $\gcd(\lambda, \mu) = 1$.  
Letting $S = \mu S_1 + \lambda S_2$, by~\cite[Theorem~8.2]{numerical} we have
$$\mathsf L_S(\lambda\mu) = \mathsf L_{S_1}(\lambda) \cup \mathsf L_{S_2}(\mu)$$
and thus
$$\max \Delta(S) = \max \Delta_S(\lambda\mu) = \min \mathsf L_{S_2}(\mu) - \max \mathsf L_{S_1}(\lambda) = p - \max \mathsf L_{S_1}(\lambda) \ge d + 1,$$
which ensures $S$ is tasty since $|\mathsf L_{S_1}(\lambda)| \ge 2$.  
\end{proof}

\begin{thm}\label{t:infiniteblandgluings}
Fix numerical semigroups $S_1$ and $S_2$.  If $S_1$ is bland and 
$$\max\Delta(S_1) \ge \max\Delta(S_2),$$
then there exists infinitely many gluings $\mu S_1 + \lambda S_2$ that are bland.
\end{thm}

\begin{proof}
Let $n_k$ denote the largest atom of $S_2$.  Choose any prime $\lambda \in S_1$ such that 
$$
\lambda > \max \Betti(S_1),
\qquad
\lambda > n_2,
\qquad \text{and} \qquad
\max \mathsf L_{S_1}(\lambda) > \max \Betti(S_2),
$$
and let $\mu = (\max \mathsf L_{S_1}(\lambda))n_k$ so that $\min \mathsf L_{S_2}(\mu) = \max \mathsf L_{S_1}(\lambda)$.  Clearly $\max \mathsf L_{S_1}(\lambda) < \lambda$, so $\gcd(\lambda,\mu) = 1$.  Letting $S = \mu S_1 + \lambda S_2$, we see $\max \Betti(S) = \lambda\mu$, ensuring
\begin{itemize}
\item 
$\mathsf L_S(\mu b) = \mathsf L_{S_1}(b)$ for each $b \in \Betti(S_1)$, 

\item 
$\mathsf L_S(\lambda b) = \mathsf L_{S_2}(b)$ for each $b \in \Betti(S_2)$, and

\item 
$\mathsf L_S(\lambda\mu) = \mathsf L_{S_1}(\lambda) \cup \mathsf L_{S_2}(\mu)$ with $\Delta_S(\lambda\mu) \le \max(\Delta(S_1) \cup \Delta(S_2))$. 

\end{itemize}
We conclude $\max \Delta(S)$ is attained at a Betti element with singleton delta set.  
\end{proof}

\begin{thm}\label{t:singleblandgluing}
Let $S_1 = \<2, 3\>$ and $S_2 = \<6, 9, 20\>$.  A gluing $S = \mu S_1 + \lambda S_2$ is bland if and only if $\lambda = 4$ and $\mu = 27$.  
\end{thm}

\begin{proof}
Note that $\Betti(S) = \{6\mu, 18\lambda, 60\lambda, \lambda\mu\}$, and that $\Delta_S(6\mu)$, $\Delta_S(18\lambda)$, and $\Delta_S(60\lambda)$ each contain $1$.  

First, suppose $\mu > 61$.  In this case, $\mu\lambda > 60\lambda > 18\lambda$, so $\max \Delta(S) \ge \max \Delta(S_2) = 4$, and $1 \in \Delta_S(\lambda\mu)$ since $1 \in \Delta_{S_2}(\mu)$, so $S$ must be tasty.  

Next, suppose $\mu \leq 60$ and $\lambda > 33$.  In this case, $1 \in \Delta_{S_1}(\lambda)$ and thus $1 \in \Delta_S(\lambda\mu)$, so $S$ is tasty if and only if $\max\Delta(S) > 1$.  We have $\mathsf L_{S_2}(\mu) \subseteq [2,10] \cap \ZZ$, and since $\lambda > 33$, $\min \mathsf L_{S_1}(\lambda) \geq 12$. From this, we conclude $\max \Delta_S(\lambda\mu) \ge 2$ and thus $S$ is tasty. 

At this point, only finitely many gluings remain, and an exhaustive computation with the \texttt{GAP} package \texttt{numericalsgps}~\cite{numericalsgpsgap} completes the proof.  
\end{proof}

\begin{example}\label{e:singleblandgluing}
We briefly examine the exceptional case identified in Theorem~\ref{t:singleblandgluing}.  Let~$S_1 = \<2, 3\>$, $S_2 = \<6, 9, 20\>$, $\lambda = 4$, and $\mu = 27$, and let 
$$S = \mu S_1 + \lambda S_2 = \<2\mu, 3\mu, 6\lambda, 9\lambda, 20\lambda\>.$$
From $\lambda\mu = 6\mu + 6\lambda + 2\mu$, one can can readily verify
$$\mathsf L_S(60\lambda) = \mathsf L_{S_1}(\lambda) \cup (\mathsf L_S(6\mu) + 2) = \{3, 7, 8, 9, 10\} \cup \{4,5,6,7\}$$
and
$$\mathsf L_S(\lambda\mu) = \mathsf L_{S_1}(\lambda) \cup \mathsf L_{S_2}(\mu) = \{2,4\} \cup \{3\},$$
resulting in $\Delta(S) = \{1\}$ and ensuring $S$ is bland.  
\end{example}

\begin{thm}\label{t:noblandgluings}
Every gluing $S = \mu S_1 + \lambda S_2$ of $S_1 = \<2, 3\>$ and $S_2 = \<6, 9, 26\>$ is tasty. 
\end{thm}

\begin{proof}
We begin by noting $\Betti(S_2) = \{18,78\}$, $\Delta_{S_2}(18) = \{1\}$, and $\Delta_{S_2}(78) = \{1,6\}$, and that every $n \in S_2$ with $n > 73$ has $|\mathsf L_{S_2}(n)| \ge 2$.  

First, suppose $\mu > 78$.  It is easy to show that $1$ lies in $\Delta_S(b)$ for each $b \in \Betti(S)$.  Since $1 \in \Delta_{S_2}(\mu)$ implies $1 \in \Delta_S(\mu\lambda)$, and since $\lambda\mu$ is the largest Betti element, $1$ lies in the delta set of each $b \in \Betti(S)$. Furthermore, $\max \Delta_{S}(78\lambda) = 6$, so $S$ is tasty.

Next, suppose $\mu \le 78$ and $\lambda > 42$.  We have $\min\mathsf L_{S_1}(\lambda) \ge 15$ and $\mathsf L_{S_2}(\mu) \subseteq [2,13] \cap \ZZ$, so $\max \Delta_S(\lambda\mu) \ge 2$, and $1 \in \Delta_S(\lambda\mu)$ since $1 \in \Delta_{S_1}(\mu)$.  This again implies $S$ is tasty. 

Now, all remaining gluings satisfy $\lambda \le 42$ and $\mu \le 78$, and an exhaustive computation using~\cite{numericalsgpsgap} verifies each is tasty.  
\end{proof}

Our final result of this section concerns gluings $\mu S + \lambda S$ of a numerical semigroup~$S$ with itself.  In the case $\mathsf e(S) = 2$, we identify the asymptotic proportion of gluings $\mu S + \lambda S$ that are tasty (Theorem~\ref{t:tastyproportion}).  The plots in Figure~\ref{f:selfglueplots} depicts the choices of $\lambda$ and $\mu$ that are assured tasty or bland by Proposition~\ref{p:blandtastyregions}.

\begin{figure}
\begin{center}
\includegraphics[height=2.9in]{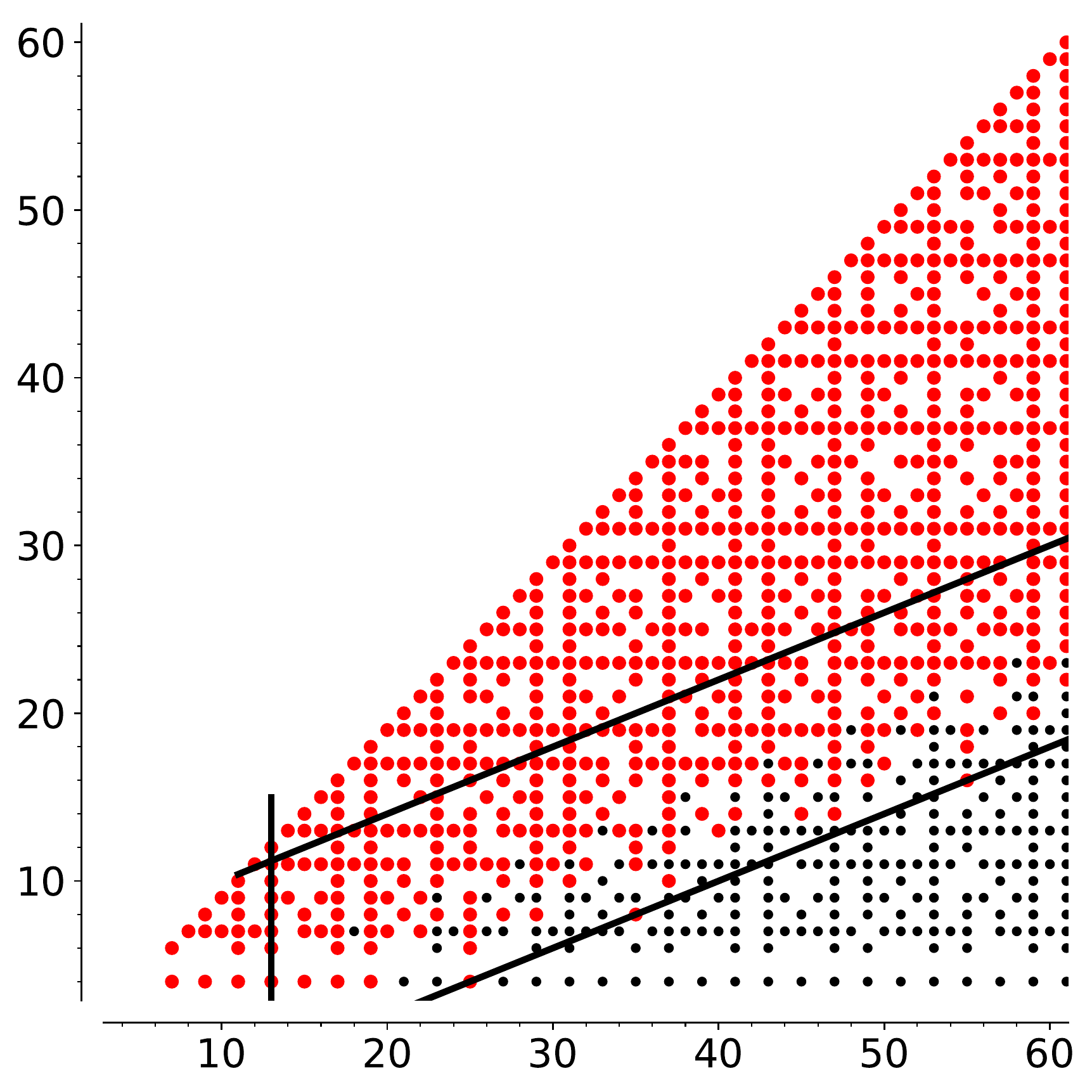}
\hspace{0.1in}
\includegraphics[height=2.9in]{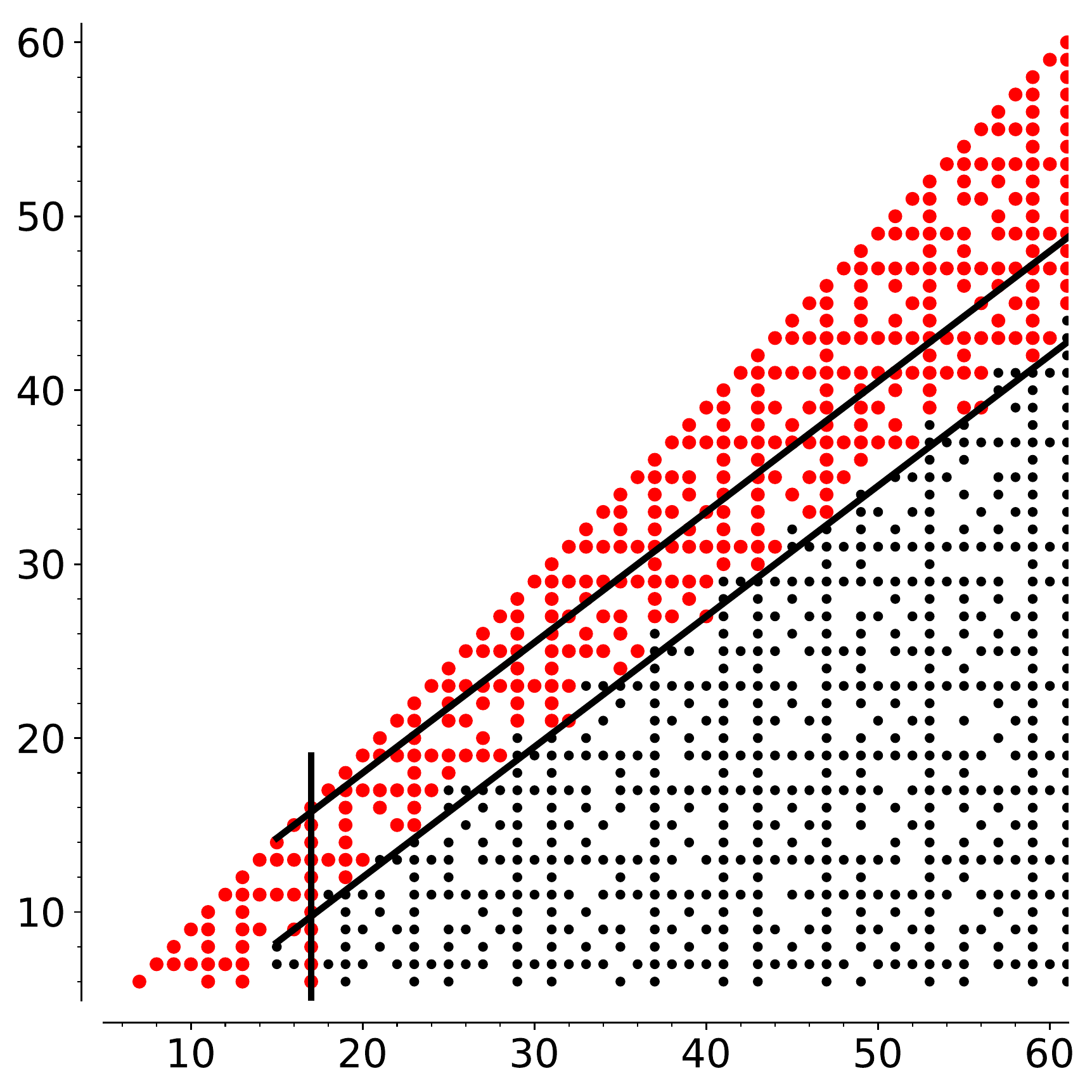}
\end{center}
\caption{Plots for $S = \<2,5\>$ (left) and $S = \<3,4\>$ (right) with a large (red) dot at $(\lambda,\mu)$ if $\mu S + \lambda S$ is tasty, and a small (black) dot if $\mu S + \lambda S$ is bland.  Each linear constraint in Proposition~\ref{p:blandtastyregions} is also depicted.}
\label{f:selfglueplots}
\end{figure}

\begin{prop}\label{p:blandtastyregions}
Fix $S = \<n_1, n_2\>$ and integers $\lambda > \mu > 2n_1n_2 - n_1 - n_2$ with $\gcd(\lambda, \mu) = 1$, and denote the gluing $G = \mu S + \lambda S$.  
\begin{enumerate}[(a)]
\item 
If $\lambda > \frac{n_2}{n_1}\mu + n_2(n_2 - n_1)$, then $G$ is tasty.

\item 
If $\lambda < \frac{n_2}{n_1}\mu - n_2(n_2 - n_1)$, then $G$ is bland.

\end{enumerate}
\end{prop}

\begin{proof}
The lower bound on $\lambda$ and $\mu$ ensures $|\mathsf L_S(\lambda)| \ge 2$ and $|\mathsf L_S(\mu)| \ge 2$, and that $\lambda\mu$ is the largest Betti element of $G$.  
If $\lambda > \frac{n_2}{n_1}\mu + n_2(n_2 - n_1)$, then 
$$\min\mathsf L_S(\lambda) \ge \tfrac{1}{n_2}\lambda > \tfrac{1}{n_1}\mu + (n_2 - n_1) \ge \max\mathsf L_S(\mu) + \max\Delta(S),$$
and thus $\max\Delta(G) = \max\Delta_G(\lambda\mu) > \max\Delta(S)$.  From this, we conclude $G$ is tasty.  
If, on the other hand, $\lambda < \frac{n_2}{n_1}\mu - n_2(n_2 - n_1)$, then
$$\min\mathsf L_S(\lambda) - \max\mathsf L_S(\mu) < (\tfrac{1}{n_2}\lambda + (n_2 - n_1)) - (\tfrac{1}{n_1}\mu - (n_2 - n_1))  < n_2 - n_1,$$
and similarly
$$\max\mathsf L_S(\lambda) - \min\mathsf L_S(\mu) > \tfrac{1}{n_1}\lambda - \tfrac{1}{n_2}\mu - 2(n_2 - n_1) > 0.$$
From these inequalities, we conclude there exists $\delta \in \Delta_G(\lambda\mu)$ with $0 < \delta < n_2 - n_1$, and since $\mathsf L_S(\lambda\mu) = \mathsf L_{S_1}(\lambda) \cup \mathsf L_{S_2}(\mu)$, we must have $\Delta_G(\lambda\mu) = \{\delta, n_2 - n_1\}$.  It follows that $\max\Delta(G) = n_2 - n_1$, so $\Delta_G(\lambda n_1n_2) = \{n_2 - n_1\}$ implies $G$ is bland. 
\end{proof}

\begin{thm}\label{t:tastyproportion}
Given a numerical semigroup $S = \<n_1, n_2\>$, the proportion of tasty gluings of $S$ with itself is $n_1/n_2$.  More precisely, 
$$\lim_{N \to \infty} \frac{\text{\# tasty gluings $\mu S + \lambda S$ with $\lambda, \mu < N$}}{\text{\# gluings $\mu S + \lambda S$ with $\lambda, \mu < N$}} = \frac{n_1}{n_2}.$$
\end{thm}

\begin{proof}
Let $L$ denote the limit on the left hand side above.  Writing
$$A = \{(a, b) \in \ZZ^2: \gcd(a, b) = 1\}$$
and $N_0 = 2n_1n_2 - n_1 - n_2$, Proposition~\ref{p:blandtastyregions}(a) yields a lower bound on $L$ of 
\begin{align*}
L
&\ge \lim_{N \to \infty} \frac{\text{\# $(\lambda,\mu) \in A \cap [N_0,N]^2$ with $\lambda > \mu$ and $\lambda > \frac{n_2}{n_1}\mu + n_2(n_2 - n_1)$}}{\text{\# $(\lambda,\mu) \in A \cap [2,N]^2$ with $\lambda > \mu$}} \\
&= \lim_{N \to \infty} \frac{\text{\# $(\lambda,\mu) \in [2,N]^2$ with $\lambda > \mu$ and $\lambda > \frac{n_2}{n_1}\mu + n_2(n_2 - n_1)$}}{\text{\# $(\lambda,\mu) \in [2,N]^2$ with $\lambda > \mu$}} \\
&= \lim_{N \to \infty} \frac{\text{\# $(\lambda,\mu) \in [2,N]^2$ with $\tfrac{n_1}{n_2}\lambda > \mu$}}{\text{\# $(\lambda,\mu) \in [2,N]^2$ with $\lambda > \mu$}}
= \frac{n_1}{n_2}
\end{align*}
where the first equality follows from \cite[Chapter~IV, Theorem~1]{lehmerasymptotic}.  
Proposition~\ref{p:blandtastyregions}(b) yields an identical upper bound for $L$, so we conclude $L = n_1/n_2$.  
\end{proof}




\end{document}